\documentclass[11pt,reqno]{amsart}
\usepackage{amssymb,mathrsfs,graphicx}
\usepackage[margin=1.15in]{geometry}
\usepackage{ifthen}
\usepackage{colortbl}
\definecolor{black}{rgb}{0.0, 0.0, 0.0}
\definecolor{red}{rgb}{1.0, 0.5, 0.5}
\provideboolean{shownotes} % define a boolean variable
\setboolean{shownotes}{true} % defined as ``true'' or ``false''
\newcommand{\margnote}[1]{
\ifthenelse{\boolean{shownotes}}%
{\marginpar{\raggedright\tiny\texttt{#1}}}%
{}%
}
\newcommand{\hole}[1]{
\ifthenelse{\boolean{shownotes}}%
{\begin{center} \fbox{ \rule {.25cm}{0cm} \rule[-.1cm]{0cm}{.4cm}
\parbox{.85\textwidth}{\begin{center} \texttt{#1}\end{center}} \rule
{.25cm}{0cm}}\end{center}} {} }
\title[Exponential synchronization of Kuramoto oscillators with time delayed coupling]{Exponential synchronization of Kuramoto oscillators with time delayed coupling}

\author[Choi]{Young-Pil Choi}
\address[Young-Pil Choi]{\newline Department of Mathematics, Yonsei University,
    \newline Seoul 03722, Republic of Korea}
\email{ypchoi@yonsei.ac.kr}

\author[Pignotti]{Cristina Pignotti}
\address[Cristina Pignotti]{\newline Dipartimento di Ingegneria e Scienze dell'Informazione e Matematica, Universit\`a di L'Aquila,
\newline 67010 L'Aquila, Italy }
\email{pignotti@univaq.it}

\numberwithin{equation}{section}

\newtheorem{theorem}{Theorem}[section]
\newtheorem{lemma}{Lemma}[section]

\newtheorem{proposition}{Proposition}[section]
\newtheorem{remark}{Remark}[section]
\newtheorem{definition}{Definition}[section]
\def\({\begin{eqnarray}}
\def\){\end{eqnarray}}
\def\[{\begin{eqnarray*}}
\def\]{\end{eqnarray*}}

\newcommand{\R}{\mathbb R}

\newcommand{\mc}{\mathcal C}

\newcommand{\bq}{\begin{equation}}
\newcommand{\eq}{\end{equation}}

\newcommand{\lt}{\left}
\newcommand{\rt}{\right}

\newcommand{\md}{\mathcal{D}}

\def\N{\mathbb{N}}

\begin{document}
%%%%%%%%%%%%%%%%
\allowdisplaybreaks

\date{\today}

\subjclass[]{}

\begin{abstract}
We discuss the asymptotic frequency synchronization for the non-identical Kuramoto oscillators with time delayed interactions. We provide explicit lower bound on the coupling strength and upper bound on the time delay in terms of initial configurations ensuring exponential synchronization. This generalizes not only the frequency synchronization estimate by Choi et al. [{\it Physica D}, 241, (2012), 735--754]  for the non-identical Kuramoto oscillators without time delays but also improves previous result by Schmidt et al. [{\it Automatica}, 48, (2012), 3008--3017] in the case of homogeneous time delays where the initial phase diameter is assumed to be less than $\pi/2$. The proof relies on a Lyapunov functional approach.
\end{abstract}

\maketitle \centerline{\date}

%\tableofcontents

%%%%%%%%%%%%%%%%%%%%%%%%%%%%%%%%%%%%%%%%%%%%%%%%%%%%%%%%%%%%%%%%%%%%%%%%%%%%%%%%%
%
%
%                        Section: Introduction
%
%
%%%%%%%%%%%%%%%%%%%%%%%%%%%%%%%%%%%%%%%%%%%%%%%%%%%%%%%%%%%%%%%%%%%%%%%%%%%%%%%%%
\section{Introduction}
Complex dynamical systems have recently received immense research interest. These systems extensively appear in many disciplines such as biology, applied mathematics, control theory, and statistical physics \cite{ABPRS, BB, PLR, Ward}. In the current work, we are interested in the synchronization phenomena of large weakly coupled Kuramoto oscillators \cite{Ku75}. More precisely, we consider the effects of time delay on the dynamics of the Kuramoto model. Let $\theta_i = \theta_i(t)$ be the phase of $i$-th Kuramoto oscillators at time $t>0$. Then,  our main system reads
\bq\label{main_eq}
\frac{d}{dt}\theta_i(t) = \Omega_i + \frac\kappa N \sum_{k\ne i} \sin(\theta_k(t - \tau) - \theta_i(t)), \quad i = 1,\dots, N, \quad  t > 0,
\eq
subject to the initial data:
\bq\label{IC}
\theta_i(s) = \theta^0_i(s) \quad \mbox{for} \quad s \in [0,-\tau],
\eq
where $\tau>0$ is the time delay and $\theta_i^0\in \mc^1[-\tau, 0],$ $ i=1,\dots, N.$ Here $\kappa > 0$ denotes the uniform coupling strength between oscillators, and $\Omega_i$ is the natural frequency of $i$-th oscillator, which is assumed to be a random variable extracted from a given distribution $g = g(\Omega)$.

For the system \eqref{main_eq}, we study an exponential frequency synchronization, which refers to the phenomenon where all oscillators have the same frequency exponentially fast as time goes on. We provide sufficient conditions for the coupling strength and the time delay in order to deduce the synchronization result. An explicit lower bound for the coupling strength, which only depends on the diameters of natural frequencies and initial phase configurations, is given in {\bf (H2)}, and upper bound on the size of the time delay is presented in {\bf (H3)} below. Papers most closely related to our work are that of Schmidt et al \cite{SPMA12}  and Choi et al \cite{CHJK12}. In \cite{SPMA12}, the frequency synchronization phenomenon in networks of non-identical Kuramoto oscillators with heterogeneous delayed coupling is discussed. Although we do not take into account networks in the system \eqref{main_eq}, the stability regime we obtain here is larger than that of \cite{SPMA12}. To be more precise, the stability regime for the frequency synchronization in \cite{SPMA12} is a subset of
\[
\lt\{ (\theta_1, \dots, \theta_N) : \max_{1 \leq i,j \leq N} |\theta_i - \theta_j| < \frac\pi2 \rt\},
\]
while our result can cover the case $\max_{1 \leq i,j \leq N} |\theta_i - \theta_j| > \pi/2$, see {\bf (H1)} below. In order to have a larger stability regime, we take a similar strategy used in \cite{CHJK12}, where the classical Kuramoto model without time delayed coupling, i.e., the system \eqref{main_eq} with $\tau = 0$, is considered. We first show that the diameter of phase configurations is uniformly bounded in time by its initial one, which is greater than $\pi/2$. Then we show that the phases will be confined inside an arc with geodesic length strictly less than $\pi/2$ in a finite time. The main difficulties in analyzing the system \eqref{main_eq} are the nonlinearity of the interaction term and the lack of conservation of mass. In this respect, we extend the result \cite{CHJK12} to the time delayed coupling case, see Remark \ref{rmk_nodelay} for more detailed discussion. Compared to the work \cite{SPMA12}, we employ a Lyapunov functional approach that gives the information about the convergence rate, which is exponential. Moreover our estimates are uniform with respect to the number of oscillators $N$. This enables us to extend results at the particle level to the continuum model, see \cite{BCM15, CCHKK14, CH17, CPP, CP19}. We also refer to \cite{N01, PJM10, YS99} for the study of effects of time delays in multi-agent dynamical system.

For the complete frequency synchronization estimate, we introduce several notations to be used throughout the paper:
\[
D(\theta(t)) :=\max_{1 \leq i,j \leq N}|\theta_i(t) - \theta_j(t)|, \quad D(\omega(t)) :=\max_{1 \leq i,j \leq N}|\omega_i(t) - \omega_j(t)|,
\]
where $\omega_i(t) := \dot\theta_i(t):=\frac{d\theta_i(t)}{dt}$,
and
\[
D(\Omega) :=\max_{1 \leq i,j \leq N}|\Omega_i - \Omega_j|.
\]
Let us denote
\begin{equation}\label{erre}
R_\omega:=\max_{t\ge -\tau}
\max_{1 \leq i \leq N} |\omega_i(t)|.
\end{equation}
Note that we can easily find
\[
|\omega_i(t)| \leq |\Omega_i| + \kappa \leq \max_{1 \leq i \leq N} |\Omega_i| + \kappa,
\]
for all $i =1,\cdots, N.$  Before stating our main result, we introduce a definition of complete frequency synchronization.
\begin{definition}\label{def_sync} Let $\theta(t) := (\theta_1(t), \cdots, \theta_N(t))$ be a global classical solution to the system \eqref{main_eq}-\eqref{IC}. Then the system exhibits the complete frequency synchronization if and only if the relative frequency differences go to zero:
\[
\lim_{t \to \infty}\max_{1 \leq i,j \leq N}|\omega_i(t) - \omega_j(t)| = 0.
\]
\end{definition}
We restrict our analysis to the case of more than $2$ oscillators, i.e. $N>2;$ for the case $N=2$ see Remark \ref{N2} below. We also list our main assumptions for our  result as follows.
\begin{itemize}
\item[{\bf (H1)}] The diameter of the initial phases satisfies
\[
D_{\theta_0}:=D(\theta(0)) \in \lt( \frac\pi2,\,\pi - \delta\rt) \quad \mbox{with }  \delta \in \lt(0,\frac\pi2\rt),
\]
\item[{\bf (H2)}] The coupling strength $\kappa >0$ is large enough such that
\[
\kappa > \lt(\frac {N} {N-2}\rt)  \frac{D(\Omega)}{\displaystyle 2\sin\lt(\frac{D_{\theta_0}}{2}\rt)\cos\lt( \frac{D_{\theta_0} + \delta}{2}\rt)},
\]
\item[{\bf (H3)}] The time delay $\tau \geq 0$ is small enough such that
\[
\tau < \bar\tau:= \min\lt\{\frac{\delta}{2R_\omega}, \frac{1}{R_\omega}\lt(\frac\pi2 - D_* \rt)\rt\},
\]
where $R_\omega$ is the constant defined in \eqref{erre} and $D_*$ is the dual angle of the initial phase diameter $D_{\theta_0}$ given by
\begin{equation}\label{Dstar}
D_* := \arcsin\lt(\sin (D_{\theta_0}) \rt) \in \lt(0,\frac\pi2\rt).
\end{equation}
\end{itemize}
Note that {\bf (H2)} and {\bf (H3)} imply
\begin{equation}\label{used}
\kappa >\lt(\frac {N} {N-2}\rt)\frac{D(\Omega)}{\displaystyle 2\sin\lt(\frac{D_{\theta_0}}{2}\rt)\cos\lt( \frac{D_{\theta_0}}{2} + R_\omega\tau\rt)}.
\end{equation}
We now present our main result of this paper.
\begin{theorem}\label{thm_main} Let $N>2$ and $\theta = \theta(t)$ be a global classical solution to the system \eqref{main_eq}-\eqref{IC} satisfying {\bf (H1)}-{\bf (H3)}. If the time delay $\tau \geq 0$ satisfies the additional smallness assumption:
\bq\label{add_as}
\tau <\ln \left ( 1+ \frac {\cos(D_* + R_\omega \bar\tau)} {\kappa (2+\cos(D_* + R_\omega \bar\tau))} \right ),
\eq
then the time delayed Kuramoto oscillators achieve the asymptotic complete frequency synchronization in the sense of Definition \ref{def_sync} exponentially fast. More precisely, we have
\[
D(\omega(t)) \leq C e^{-\gamma t}, \quad t \geq t_*,
\]
for some $t_*>0$ and positive constants $\gamma$ and $C$ independent of $t$.
\end{theorem}
\begin{remark}For the case of identical oscillators, i.e., the distribution function $g$ for the natural frequencies is the form of the Dirac measure on $\R$ given unit mass to the point $\Omega_0 \in \R$, $g(\Omega) = \delta_{\Omega_0}(\Omega)$, it is clear the corresponding diameter $D(\Omega) = 0$. This implies that the assumption {\bf (H2)} is not required for the complete frequency synchronization estimate in Theorem \ref{thm_main}.
\end{remark}

\begin{remark}Differentiating the system \eqref{main_eq} with respect to $t$, we obtain
\[
\frac{d}{dt} \omega_i(t) = \frac\kappa N\sum_{k \ne i} \cos(\theta_k(t-\tau) - \theta_i(t)) (\omega_k(t-\tau) - \omega_i(t)), \quad i = 1,\dots, N.
\]
Under our main assumptions {\bf (H1)}-{\bf (H3)}, we also find there exists $t_* > 0$ such that
\[
\lt|\theta_k(t-\tau) - \theta_i(t)\rt| \leq D_* + R_w\tau < D_* + R_w \bar\tau \le  \frac\pi2.
\]
for  $t>t_* $, see \eqref{estimate_below_bis}. Then one can apply \cite[Lemma 3.5]{SPMA12} to conclude the complete frequency synchronization estimate without further assuming \eqref{add_as}. However, in this way, we cannot obtain the exponential decay of the frequency diameter $D(\omega(t))$.
\end{remark}
\begin{remark}\label{rmk_cri2} If $D_{\theta_0} \in (0, \pi/2)$, then $D_* = D_{\theta_0}$ and thus Theorem \ref{thm_main} holds replacing $D_*$ by $D_{\theta_0}$. On the other hand, if $D_{\theta_0} = \pi/2$, then we can show that there exist $\eta$ and $t_{**}>0$ such that
\[
D(\theta(t)) < D_* =\frac\pi2 \quad \mbox{for} \quad t \geq t_{**},
\]
see Remark \ref{rmk_cri} for details. In this case, we can also show the exponential decay of $D(\omega(t))$ under smallness assumptions on the time delay $\tau$; however, it is not clear to express the explicit bound for $\tau$. This is the reason why we restrict the diameter of the initial phases to the interval $(\pi/2,\pi-\delta)$, not including $\pi/2$.
\end{remark}
\begin{remark}\label{rmk_nodelay}If there is no effects of time delay, i.e., $\tau = 0$, then the assumptions {\bf (H3)} and \eqref{add_as} can be removed, and the constant $\delta$ appeared in {\bf (H1)} and {\bf (H2)} can be zero. Furthermore the constant $N/(N-2)$ can be the unity. Then we have the following reduced assumptions:
\begin{itemize}
\item[{\bf (H1)$^\prime$}] The diameter of the initial phases satisfies
\[
D_{\theta_0} \in \lt( \frac\pi2,\,\pi\rt),
\]
\item[{\bf (H2)$^\prime$}] The coupling strength $\kappa >0$ satisfies
\[
\kappa > \frac{D(\Omega)}{\displaystyle 2\sin\lt(\frac{D_{\theta_0}}{2}\rt)\cos\lt( \frac{D_{\theta_0}}{2}\rt)} =  \frac{D(\Omega)}{\sin D_{\theta_0}}.
\]
\end{itemize}
In view of Remark \ref{rmk_cri2}, the assumption {\bf (H1)$^\prime$} can be relaxed as $D_{\theta_*} \in \lt(0,\,\pi\rt)$. This gives exactly the same assumptions for the complete frequency synchronization estimate in \cite{CHJK12}. We refer to \cite{BCM15, CHKK13, CS09, DB11, HKR16} for the complete frequency synchronization estimate in the case of non time delayed interactions.
\end{remark}

The rest of this paper is organized as follows. In Section \ref{sec2}, we estimate the diameter of phase configurations. Under the assumptions  {\bf (H1)}-{\bf (H3)}, we show that the phase diameter becomes less than $\pi/2$ in a finite time. Finally, in Section \ref{sec3}, by constructing a Lyapunov functional we provide the details of proof of Theorem \ref{thm_main} showing the exponential frequency synchronization.

%%%%%%%%%%%%%%%%%%%%%%%%%%%%%%%%%%%%%%%%%%%%%%%%%%%%%%%%%%%%%%%%%%%%%%%%%%%%%%%%%
%
%
%                       \section{Uniform bound estimate of the diameter of phase configurations}
%
%
%%%%%%%%%%%%%%%%%%%%%%%%%%%%%%%%%%%%%%%%%%%%%%%%%%%%%%%%%%%%%%%%%%%%%%%%%%%%%%%%%
\section{Uniform bound estimate of the diameter of phase configurations}\label{sec2}

In this section, we present the uniform-in-time bound estimate of the phase diameter $D(\theta(t))$ for the time delayed Kuramoto oscillators \eqref{main_eq}. Note that $D(\theta(t))$ is not $\mc^1$ in general, and thus we use the upper Dini derivative defined by
\[
\md^+ F(t):= \limsup_{h \to 0+} \frac{F(t+h) - F(t)}{h} \quad \mbox{for a given function $F(t)$}
\]
to take into account the time derivative of the phase diameter function.

\begin{lemma}\label{lem_ineq_d} Let $N>2$ and $\theta = \theta(t)$ be a global classical solution to the system \eqref{main_eq}-\eqref{IC} satisfying {\bf (H1)}-{\bf (H3)}. Then we obtain
\[
D(\theta(t)) \leq  D_{\theta_0} \quad \mbox{for} \quad t \geq 0.
\]
Moreover, we have
\bq\label{ineq_d}
\md^+ D(\theta(t)) \leq D(\Omega) - 2\kappa  \lt(\frac {N-2}N\rt)\sin\lt(\frac{D(\theta(t))}{2} \rt) \cos \lt( \frac{D(\theta(t))}{2} + R_\omega \tau \rt)
\eq
for almost everywhere $t \geq 0$ for which $D(\theta(t))> \delta/2.$
\end{lemma}
\begin{proof} {\bf $\bullet$ (Step A)} Due to the continuity of the phase trajectories $\theta_i(t)$,
there is an at most countable system of open, mutually disjoint
intervals $\{\mathcal{I}_\sigma\}_{\sigma\in\N}$ such that
$$
   \bigcup_{\sigma\in\N} \overline{\mathcal{I}_\sigma} = [0,\infty)
$$
and for each ${\sigma\in\N}$ there exist indices $i = i(\sigma)$ and $j = j(\sigma)$ such that
$$
D(\theta(t)) = \theta_i(t) - \theta_j(t) \quad\mbox{for } t\in \mathcal{I}_\sigma.
$$
A straightforward computation gives
\begin{align}\label{est_d}
\begin{aligned}
\md^+ D(\theta(t))\big|_{t=0} &= \md^+(\theta_i(t) - \theta_j(t) )\big\vert_{t=0} \cr
&= \Omega_i  - \Omega_j + \frac\kappa N \sum_{k\ne i} \sin(\theta_k(- \tau) - \theta_i(0)) - \sum_{k\ne j}\sin(\theta_k(- \tau) - \theta_j(0)) \cr
&=\Omega_i  - \Omega_j + \frac\kappa N \sum_{k\ne i, j} \lt(\sin(\theta_k(- \tau) - \theta_i(0)) - \sin(\theta_k(- \tau) - \theta_j(0))\rt) \cr
&\quad
+\frac\kappa N\lt(
\sin(\theta_j(-\tau)-\theta_i(0))-\sin(\theta_i(-\tau)-\theta_j(0))
\rt)\cr
&= \Omega_i  - \Omega_j - \frac{2\kappa}{N} \sum_{k\ne i, j} \cos\lt(\theta_k(- \tau) - \frac{\theta_i(0) + \theta_j(0)}{2} \rt)\sin\lt(\frac{D_{\theta_0}}{2} \rt) \cr
&\quad
+\frac\kappa N\lt(
\sin(\theta_j(-\tau)-\theta_i(0))-\sin(\theta_i(-\tau)-\theta_j(0))
\rt).
\end{aligned}
\end{align}
On the other hand, we find
\bq\label{richiamo}
\theta_k(- \tau) = \theta_k(0) - \int_{-\tau}^0 \frac{d}{ds}\theta_k(s)\,ds \quad \mbox{and} \quad  \lt|\int_{-\tau}^0 \frac{d}{ds}\theta_k(s)\,ds\rt| \leq R_\omega \tau,
\eq
due to \eqref{erre}. Thus we obtain
\[
\lt| \theta_k(-\tau) - \frac{\theta_i(0) + \theta_j(0)}{2} \rt| \leq \frac{D_{\theta_0}}{2} + R_\omega \tau \quad \mbox{for all} \quad k \in \{1,\dots,N\}.
\]
It follows from {\bf (H1)} and {\bf (H2)} that
\[
\frac{D_{\theta_0}}{2} + R_\omega \tau < \frac{\pi - \delta}{2} + R_\omega\tau < \frac\pi2,
\]
and this asserts
\bq\label{as01}
\cos\lt(\theta_k(- \tau) - \frac{\theta_i(0) + \theta_j(0)}{2} \rt) \geq \cos \lt( \frac{D_{\theta_0}}{2} + R_\omega \tau \rt).
\eq
Thus we get
$$\begin{aligned}
&- \frac{2\kappa}{N} \sum_{k\ne i, j} \cos\lt(\theta_k(- \tau) - \frac{\theta_i(0) + \theta_j(0)}{2} \rt)\sin\lt(\frac{D_{\theta_0}}{2} \rt) \cr
&\qquad \leq - 2\kappa \lt(\frac {N-2}N\rt) \sin\lt(\frac{D_{\theta_0}}{2} \rt) \cos \lt( \frac{D_{\theta_0}}{2} + R_\omega \tau \rt).
\end{aligned}$$
Moreover, we observe that
\begin{align}\label{additional}
\begin{aligned}
&\sin(\theta_j(-\tau)-\theta_i(0))-\sin(\theta_i(-\tau)-\theta_j(0))\cr
&\quad = - 2\cos\lt( \frac {\theta_i(-\tau)-\theta_i(0)} 2 +  \frac {\theta_j(-\tau)-\theta_j(0)} 2\rt)\,
\sin \lt(\frac {\theta_i(-\tau )-\theta_j(-\tau )} 2 +  \frac {\theta_i(0)-\theta_j(0)} 2 \rt).
\end{aligned}
\end{align}
Now, we get from \eqref{richiamo} that
$$\left\vert
\frac {\theta_i(-\tau)-\theta_i(0)} 2 +  \frac {\theta_j(-\tau)-\theta_j(0)} 2\right\vert \le \frac {R_\omega\tau} 2 + \frac {R_\omega\tau} 2=R_\omega\tau <\frac \delta 2,$$
which implies
\bq\label{cosinus}
\cos\lt( \frac {\theta_i(-\tau)-\theta_i(0)} 2 +  \frac {\theta_j(-\tau)-\theta_j(0)} 2\rt) >0.
\eq
Also, we have
$$
\frac {\theta_i(-\tau)-\theta_j(-\tau )}2 = \frac {D_{\theta_0}} 2 +\frac 12 \int_0^{-\tau}  \frac d {ds} (\theta_i(s) - \theta_j(s)) \,ds,
 $$
which gives
$$
\frac {\theta_i(-\tau)-\theta_j(-\tau )}2 + \frac {\theta_i(0)-\theta_j(0)}2 = D_{\theta_0} + \frac 12 \int_0^{-\tau} \frac d {ds} (\theta_i(s) -\theta_j(s)) \,ds.
$$
On the other hand, we find
\[
\lt| \frac 12 \int_0^{-\tau} \frac d {ds} (\theta_i(s) -\theta_j(s)) \,ds\rt| \leq R_\omega \tau < \frac\delta2 \quad \mbox{and} \quad \frac\pi2 < D_{\theta_0} < \pi - \delta.
\]
This yields
$$
0< \frac \pi 2 -\frac \delta 2 <\frac {\theta_i(-\tau)-\theta_j(-\tau )}2 + \frac {\theta_i(0)-\theta_j(0)}2 < \pi-\frac \delta 2.$$
Thus, we obtain
\bq\label{sinus}
\sin \lt(\frac {\theta_i(-\tau )-\theta_j(-\tau )} 2 +  \frac {\theta_i(0)-\theta_j(0)} 2 \rt) >0.
\eq
By using \eqref{cosinus} and \eqref{sinus} in \eqref{additional}, we deduce
$$
\sin(\theta_j(-\tau)-\theta_i(0))-\sin(\theta_i(-\tau)-\theta_j(0))<0.
$$
This and \eqref{as01}, used in \eqref{est_d}, yield
\[
\md^+ D(\theta(t))\big|_{t=0} \leq D(\Omega) - 2\kappa \lt(\frac {N-2}N\rt) \sin\lt(\frac{D_{\theta_0}}{2} \rt) \cos \lt( \frac{D_{\theta_0}}{2} + R_\omega \tau \rt) < 0,
\]
where for the last inequality we used \eqref{used}. \newline

{\bf $\bullet$ (Step B)} From {\bf (Step A)}, we have that the $D(\theta(t))$ strictly starts to decrease at $t = 0+$. If there is a $t_0 > 0$ such that
\[
D(\theta(t_0)) = D_{\theta_0} \quad \mbox{and} \quad  D(\theta(t)) < D_{\theta_0}\quad \mbox{for} \quad t < t_0,
\]
then the following must hold
\[
\md^+ D(\theta(t))\big|_{t = t_0 -} \geq 0.
\]
On the other hand, in a similar fashion as the above, at that time, we get
$$\begin{aligned}
 0 \leq \md^+ D(\theta(t))\big|_{t=t_0-} &\leq D(\Omega) - 2\kappa \lt(\frac {N-2}N\rt)\sin\lt(\frac{D(\theta(t_0))}{2} \rt) \cos \lt( \frac{D(\theta(t_0))}{2} + R_\omega \tau \rt)\cr
&= D(\Omega) - 2\kappa \lt(\frac {N-2}N\rt) \sin\lt(\frac{D_{\theta_0}}{2} \rt) \cos \lt( \frac{D_{\theta_0}}{2} + R_\omega \tau \rt) < 0.
\end{aligned}$$
This leads a contradiction. Hence we have
\[
D(\theta(t)) \leq D_{\theta_0} \quad \mbox{for} \quad t \geq 0.
\]
The differential inequality \eqref{ineq_d} just follows from the above computations until $D(\theta (t))$ remains greater than $\delta/2$.
Indeed, being $D(\theta (t))<\pi-\delta,\ t>0,$ one can obtain the analogous of \eqref{as01} and \eqref{cosinus} for each fixed  $t>0$
instead of $t=0.$ Moreover, until $D(\theta(t))> \delta/2$ one can also deduce the analogous of \eqref{sinus}, namely
\[
\sin \lt(\frac {\theta_i(t-\tau )-\theta_j(t-\tau )} 2 +  \frac {\theta_i(t)-\theta_j(t)} 2 \rt) >0.
\]
In fact, this follows from
$$
\frac {\theta_i(t-\tau)-\theta_j(t-\tau )}2 + \frac {\theta_i(t)-\theta_j(t)}2 = D(\theta(t)) + \frac 12 \int_t^{t-\tau} \frac d {ds} (\theta_i(s) -\theta_j(s))\, ds,
$$
which allows to deduce
$$
0= \frac \delta 2 -\frac \delta 2 <\frac {\theta_i(t-\tau)-\theta_j(t-\tau )}2 + \frac {\theta_i(t)-\theta_j(t)}2 < \pi-\frac \delta 2.$$
\end{proof}

In the following proposition, we show that any phase configurations satisfying the assumptions {\bf (H1)}-{\bf (H3)} will be shrink to a smaller set whose diameter is in the range $(0,\pi/2)$ in a finite time. Note that this observation is crucial to apply \cite[Lemma 3.5]{SPMA12} to have the complete frequency synchronization estimate.

\begin{proposition}\label{prop_key} Let $N>2$ and $\theta = \theta(t)$ be a global classical solution to the system \eqref{main_eq}-\eqref{IC} satisfying {\bf (H1)}-{\bf (H3)}. Then there exists $t_* \geq 0$ such that
\[
D(\theta(t)) <  D_* \quad \mbox{for all} \quad t  > t_*,
\]
where $D_*$ is the dual angle of the initial phase diameter $D_{\theta_0}$  defined in \eqref{Dstar}.
\end{proposition}
\begin{proof}
Note that $D_*>\delta$ due to {\bf (H1)}. If $D(\theta(t)) \in [D_*, D_{\theta_0}]$ for $t\geq 0$, then it follows from Lemma \ref{lem_ineq_d} that
\bq\label{est_d2}
\md^+ D(\theta(t)) \leq D(\Omega) - 2\kappa \lt(\frac {N-2}N\rt) \sin\lt(\frac{D(\theta(t))}{2} \rt) \cos \lt( \frac{D(\theta(t))}{2} + R_\omega \tau \rt).	
\eq
On the other hand, we get
\begin{align}\label{est_coupling}
\begin{aligned}
& 2\sin\lt(\frac{D(\theta(t))}{2} \rt) \cos \lt( \frac{D(\theta(t))}{2} + R_\omega \tau \rt) \cr
&\quad = 2\sin\lt(\frac{D(\theta(t))}{2} \rt) \lt( \cos\lt( \frac{D(\theta(t))}{2} \rt) \cos (R_\omega \tau) - \sin \lt( \frac{D(\theta(t))}{2}\rt) \sin (R_\omega \tau) \rt)\cr
&\quad = \sin D(\theta(t)) \cos (R_\omega \tau) - 2\sin^2\lt( \frac{D(\theta(t))}{2}\rt)  \sin (R_\omega \tau) \cr
&\quad \geq \sin D_{\theta_0} \cos (R_\omega \tau) - 2\sin^2\lt( \frac{D_{\theta_0}}{2}\rt)  \sin (R_\omega \tau) \cr
&\quad = 2\sin\lt(\frac{D_{\theta_0}}{2} \rt) \cos \lt( \frac{D_{\theta_0}}{2} + R_\omega \tau \rt),
\end{aligned}
\end{align}
where we used the fact that
\[
\sin D(\theta(t)) \geq \sin D_* = \sin D_{\theta_0} \quad \mbox{and} \quad \sin \lt( \frac{D(\theta(t))}{2}\rt) \leq \sin \lt( \frac{D_{\theta_0}}{2} \rt)
\]
for $D(\theta(t)) \in [D_*, D_{\theta_0}]$. This together with \eqref{est_d2} yields
\[
\md^+ D(\theta(t)) \leq  D(\Omega) - 2 \kappa \lt(\frac {N-2}N\rt) \sin\lt(\frac{D_{\theta_0}}{2} \rt) \cos \lt( \frac{D_{\theta_0}}{2} + R_\omega \tau \rt) < 0,
\]
due to $\eqref{used}$. In particular, the above differential inequality provides
\[
D(\theta(t)) \leq D_{\theta_0} + \lt(D(\Omega) - 2 \kappa\lt(\frac {N-2}N\rt) \sin\lt(\frac{D_{\theta_0}}{2} \rt) \cos \lt( \frac{D_{\theta_0}}{2} + R_\omega \tau \rt)\rt)t.
\]
Since the right hand side of the above inequality goes to $-\infty$ as $t \to \infty$, there should be a $t_* > 0$ such that $D(\theta(t))$ leaves the interval $[D_*, D_{\theta_0}]$ after that time $t_*$. Without loss of generality we  assume that $D(\theta (t_*)) <D_*.$ We next show that
\[
D(\theta(t)) <  D_* < \frac\pi2 \quad  \mbox{for all} \quad t \geq t_*.
\]
For this, we use the standard continuity argument. Let us define a set $\mathcal{S}$ by
\[
\mathcal{S} := \lt\{ T > t_* : D(\theta(t)) < D_* \quad \mbox{for} \quad t \in [t_*, T) \rt\}.
\]
Since $\mathcal{S} \neq \emptyset$, by continuity of $D(\theta(t))$, we can consider $T^* := \sup \mathcal{S} > 0$, and it holds $D(\theta(t)) < D_*$ for $t \in [t_*, T^*)$. We then show that $T^* = +\infty$. Suppose that $T^* < \infty$, then we find
\[
\lim_{t \to T^*-} D(\theta(t)) = D_*.
\]
In particular, there exists $t^1_*\ge t_*$ such that
$$
D(\theta (t))>\frac \delta 2 \quad  \mbox{for all} \quad t \in [t^1_*, T^*),
$$
due to the fact $D_* > \delta > \delta/2$. Note that
\[
\frac{\sin D_*}{D_*} < \frac{\sin D(\theta(t))}{D(\theta(t))} \quad \mbox{for all} \quad t \in [t_*, T^*).
\]
Combining this, \eqref{est_d2} and \eqref{est_coupling}, we obtain for $t\in [t^1_*, T^*),$
$$\begin{aligned}
&\md^+ D(\theta(t)) \cr
&\quad \leq D(\Omega) -\kappa \lt(\frac {N-2}N\rt)\lt(\sin D(\theta(t)) \cos (R_\omega \tau) - 2\sin^2\lt( \frac{D(\theta(t))}{2}\rt)  \sin (R_\omega \tau)\rt) \cr
&\quad \leq D(\Omega) -\kappa \lt(\frac {N-2}N\rt)\lt( \lt(\frac{\sin D_*}{D_*}\rt) \cos (R_\omega \tau) D(\theta(t)) - 2\sin^2\lt( \frac{D_{\theta_0}}{2}\rt)  \sin (R_\omega \tau) \rt).
\end{aligned}$$
We now apply Gr\"onwall's lemma to previous estimate to get, for every $t\in [t^1_*, T^*),$
$$\begin{aligned}
D(\theta(t)) &\leq D(\theta (t_*^1)) \exp\lt( - \kappa\lt(\frac {N-2}N\rt) \frac{\sin D_*}{D_*} \cos (R_\omega \tau)(t-t_*^1) \rt)\cr
&\quad + \frac{\displaystyle D(\Omega) + 2\kappa\lt(\frac {N-2}N\rt)\sin^2\lt( \frac{D_{\theta_0}}{2}\rt)  \sin (R_\omega \tau) }{ \displaystyle \kappa \lt(\frac {N-2}N\rt) \lt(\frac{\sin D_*}{D_*}\rt) \cos (R_\omega \tau)} \cr
&\qquad \times \lt(1 -  \exp\lt( - \kappa \lt(\frac {N-2}N\rt)  \frac{\sin D_*}{D_*} \cos (R_\omega \tau)(t-t_*^1) \rt) \rt).
\end{aligned}$$
Note that
$$\begin{aligned}
&\kappa > \lt(\frac {N} {N-2}\rt) \frac{D(\Omega)}{2\sin\lt(\frac{D_{\theta_0}}{2}\rt)\cos\lt( \frac{D_{\theta_0}}{2} + R_\omega \tau\rt)} \cr
& \Longleftrightarrow \quad  D(\Omega) < \kappa\lt(\frac {N-2}N\rt)\lt( \sin D_{\theta_0} \cos (R_\omega \tau) - 2\sin^2\lt( \frac{D_{\theta_0}}{2}\rt)  \sin (R_\omega \tau)\rt) \cr
& \Longleftrightarrow \quad D(\Omega) + 2 \kappa \lt(\frac {N-2}N\rt)\sin^2\lt( \frac{D_{\theta_0}}{2}\rt)  \sin (R_\omega \tau) < \kappa \lt(\frac {N-2}N\rt)\sin D_* \cos (R_\omega \tau) \cr
& \Longleftrightarrow \quad \frac{\displaystyle D(\Omega) + 2\kappa \lt(\frac {N-2}N\rt) \sin^2\lt( \frac{D_{\theta_0}}{2}\rt)  \sin (R_\omega \tau) }{
\displaystyle \kappa\lt(\frac {N-2}N\rt)\frac
{\sin D_*}{D_*} \cos (R_\omega \tau)} < D_*,
\end{aligned}$$
where we used $\sin D_{\theta_0} = \sin D_*$. Thus we have
$$\begin{aligned}
D(\theta(t)) & < D(\theta (t_*^1)) \exp \lt( - \kappa \lt(\frac {N-2}N\rt) \lt(\frac{\sin D_*}{D_*}\rt) \cos (R_\omega \tau)(t-t^1_*) \rt) \cr
&\quad + D_* \lt(1 -  \exp\lt( - \kappa \lt(\frac {N-2}N\rt)\lt(\frac{\sin D_*}{D_*} \rt)\cos (R_\omega \tau)(t-t^1_*) \rt) \rt)\cr
&= D_* - (D_* -  D(\theta(t^1_*)))\exp \lt( -  \kappa \lt(\frac {N-2}N\rt)\lt(\frac{\sin D_*}{D_*}\rt) \cos (R_\omega \tau)(t-t^1_*) \rt).
\end{aligned}$$
We now let $t \to T^*-$ to the above inequality to find
$$\begin{aligned}
D_* &= \lim_{t \to T^*-} D(\theta(t)) \cr
&\leq D_* - (D_* -  D(\theta (t_*^1)))\exp \lt( -  \kappa\lt(\frac {N-2}N\rt)\lt(\frac{\sin D_*}{D_*}\rt) \cos (R_\omega \tau)(T^*-t_*^1) \rt)\cr
&< D_*.
\end{aligned}$$
This is a contradiction, and thus $T^* = +\infty$. Then, we conclude that
\[
D(\theta(t)) < D_* < \frac\pi2 \quad \mbox{for all} \quad t \geq t_*,
\]
for some $t_* > 0$. This completes the proof.
\end{proof}

\begin{remark}\label{rmk_cri}  For the case $D_{\theta_0
} = \pi/2$, assume {\bf (H2)} with a constant $\delta\in (0, \frac \pi 2)$ and let us take
$$\tau<\bar \tau := \frac \delta {2R_\omega}.$$
Then, it results
$$\kappa >
\lt(\frac {N} {N-2}\rt) \frac{D(\Omega)}{\cos (R_\omega \tau) - \sin (R_\omega \tau)}.$$
We can show that there exist $\eta > 0$ and $t_{**}>0$ such that
\[
D(\theta(t)) \leq \eta < \frac\pi2 \quad \mbox{for all} \quad t \geq t_{**}.
\]
More precisely, by using the same argument as in the proof of Lemma \ref{lem_ineq_d}, we find
\[
\md^+ D(\theta(t))\big|_{t=0} < 0.
\]
This means that there exists a $0 < t_{**} \ll 1$ such that
\[
D(\theta(t_{**})) < \frac\pi2.
\]
Note that we can choose $\eta \in (D(\theta(t_{**})), \pi/2)$ such that
\[
\tan(R_\omega \tau) > \frac{1 - \sin \eta}{\cos \eta}
\]
since
\[
\lim_{\eta \to \frac\pi2-} \frac{1 - \sin \eta}{\cos \eta} = 0.
\]
On the other hand, the above inequality is equivalent to
\[
\sin(R_\omega \tau)\lt(1 - 2\sin^2 \frac\eta2\rt) > \cos(R_\omega \tau) (1 - \sin \eta),
\]
i.e.,
\[
\sin \eta \cos (R_\omega \tau) - 2\sin^2 \frac\eta2 \sin (R_\omega \tau) > \cos (R_\omega \tau) - \sin (R_\omega \tau).
\]
This again implies
$$\begin{aligned}
\kappa > \lt(\frac {N} {N-2}\rt) \frac{D(\Omega)}{\cos (R_\omega \tau) - \sin (R_\omega \tau)} &>\lt(\frac {N} {N-2}\rt)\frac{D(\Omega)}{\displaystyle \sin \eta \cos (R_\omega \tau) - 2\sin^2 \lt(\frac\eta2\rt) \sin (R_\omega \tau)}\cr
&= \lt(\frac {N} {N-2}\rt)\frac{D(\Omega)}{\displaystyle 2\sin\lt(\frac\eta2\rt)\cos\lt( \frac\eta2 + R_\omega \tau\rt)}.
\end{aligned}$$
We finally use the same argument as in Lemma \ref{lem_ineq_d} to conclude
\[
D(\theta(t)) \leq \eta < \frac\pi2 \quad \mbox{for all} \quad t \geq t_{**}.
\]
\end{remark}

\begin{remark}\label{N2}
In the case of only $2$ oscillators, instead of {\bf (H2)} one can assume
\begin{equation}\label{per2}
\kappa >\frac {D(\Omega)} {\cos (R_\omega\tau)\,\sin (D_{\theta_0}+R_\omega\tau )}.
\end{equation}
Without loss of generality, let us suppose $D_{\theta_0} = \theta_1(0)-\theta_2(0).$ Then, we have
\begin{equation}\label{prima}
\md^+ D(\theta(t))\vert_{t=0} = D(\Omega)+\frac \kappa 2 (\sin (\theta_2(-\tau)-\theta_1(0)) -\sin (\theta_1(-\tau)-\theta_2(0))).
\end{equation}
Now, observe that
\begin{equation}\label{N2uno}
\begin{array}{l}
\displaystyle {\sin (\theta_2(-\tau)-\theta_1(0)) -\sin (\theta_1(-\tau)-\theta_2(0))}\\
\displaystyle{
\ \ =- 2\cos\lt( \frac {\theta_1(-\tau)-\theta_1(0)} 2 +  \frac {\theta_2(-\tau)-\theta_2(0)} 2\rt)\,
\sin \lt(\frac {\theta_1(-\tau )-\theta_2(-\tau )} 2 +  \frac {\theta_1(0)-\theta_2(0)} 2 \rt).}
\end{array}
\end{equation}
It is easy to see that
\begin{equation}\label{11}
\cos\lt( \frac {\theta_1(-\tau)-\theta_1(0)} 2 +  \frac {\theta_2(-\tau)-\theta_2(0)} 2\rt)>\cos (R_\omega\tau).
\end{equation}
Moreover, we obtain
\begin{equation}\label{12}
\sin \lt(\frac {\theta_1(-\tau )-\theta_2(-\tau )} 2 +  \frac {\theta_1(0)-\theta_2(0)} 2 \rt) >\sin (D_{\theta_0}+R_\omega \tau).
\end{equation}
Indeed, we find
$$\frac {\theta_1(-\tau) -\theta_2(-\tau)} 2= \frac {D_{\theta_0}} 2 +\frac 12 \int_0^{-\tau} \frac d {ds} (\theta_1(s)-\theta_2(s))\, ds,$$
and then
$$\frac {\theta_1(-\tau) -\theta_2(-\tau)} 2+ \frac {\theta_1(0) -\theta_2(0)} 2 = D_{\theta_0}+\frac 12 \int_0^{-\tau} \frac d {ds} (\theta_1(s)-\theta_2(s))\, ds.$$
On the other hand, we use \eqref{erre} to obtain
\[
\lt| \frac 12 \int_0^{-\tau} \frac d {ds} (\theta_1(s) -\theta_2(s)) \,ds\rt| \leq R_\omega \tau < \frac\delta2 \quad \mbox{and} \quad \frac\pi2 < D_{\theta_0} < \pi - \delta.
\]
This yields
\begin{equation}\label{13}
 0<D_{\theta_0}- R_\omega \tau<\frac {\theta_1(-\tau)-\theta_2(-\tau )}2 + \frac {\theta_1(0)-\theta_2(0)}2 < D_{\theta_0}+R_\omega \tau <\pi .
 \end{equation}
Since, by {\bf (H3)}, $R_\omega\tau <\frac \pi 2 - D_*
,$
 it then holds
$$D_{\theta_0}- R_\omega\tau > D_{\theta_0}-\frac \pi 2 +D_*> D_*,$$
which, together with \eqref{13} gives \eqref{12}.
Therefore, using  inequalities \eqref{11} and \eqref{12} in \eqref{N2uno}, we have
$$
\sin (\theta_2(-\tau)-\theta_1(0)) -\sin (\theta_1(-\tau)-\theta_2(0))\le -2 \cos (R_\omega\tau) \sin (D_{\theta_0}+R_\omega \tau),
$$
that allows to deduce, from  \eqref{prima},
\begin{equation*}
\md^+ D(\theta(t))\vert_{t=0} \le D(\Omega) -\kappa \cos (R_\omega\tau) \sin (D_{\theta_0}+R_\omega \tau).
\end{equation*}
So, thanks to \eqref{per2}, one can argue similarly to the case of $N>2$ oscillators proving that
$$D(\theta (t))<D_* \quad\mbox{for}\quad t>t_*,$$
for a suitable time $t_*>0.$
\end{remark}

%%%%%%%%%%%%%%%%%%%%%%%%%%%%
%
%
%\section{Asymptotic frequency synchronization estimate: Proof of Theorem \ref{thm_main}}
%
%
%
%%%%%%%%%%%%%%%%%%%%%%%%%%%%%%%%%%
\section{Asymptotic frequency synchronization estimate: Proof of Theorem \ref{thm_main}}\label{sec3}
In this section, we provide the details of the proof of Theorem \ref{thm_main}. We estimate the frequency diameter function $D(\omega(t))$ to show the exponential frequency synchronization. For this, we differentiate the system \eqref{main_eq} with respect to time $t$ to find
\begin{equation}\label{eq_omega}
\frac{d}{dt} \omega_i(t) = \frac\kappa N\sum_{k \ne i} \cos(\theta_k(t-\tau) - \theta_i(t)) (\omega_k(t-\tau) - \omega_i(t)), \quad i = 1,\dots, N, \quad t > 0.
\end{equation}
It follows from Proposition \ref{prop_key} that there exists $t_* > 0$ such that $D(\theta(t)) \leq D_*$ for $t \geq t_*$, where $D_* \in (0,\pi/2)$ is given by $\sin D_* = \sin D_{\theta_0}$. This implies
\[
\lt|\theta_k(t-\tau) - \theta_i(t)\rt| = \lt|\theta_k(t) - \theta_i(t) - \int_{t - \tau}^t \frac{d}{ds}\theta_k(s)\,ds\rt| \leq D(\theta(t)) + R_w \tau \leq D_* + R_w\tau
\]
for $t\geq t_*$ and any $1 \leq i,k \leq N$. On the other hand, by ({\bf H3}), we get
\begin{equation}\label{estimate_below_bis}
\lt|\theta_k(t-\tau) - \theta_i(t)\rt| \leq D_* + R_w\tau < D_* + R_w \bar\tau\le \frac\pi2, \quad t\ge t_*>0.
\end{equation}
Thus, by writing $\zeta_* = \cos(D_* + R_w \tau  ) > 0$, we have
\begin{equation}\label{estimate_below}
\cos(\theta_k(t-\tau) - \theta_i(t)) > \zeta_* > 0, \quad t\ge t_*>0
\end{equation}
for $1 \leq i, k \leq N$. With the above observation, we provide the Gr\"onwall-type inequality for $D(\omega(t))$ in the lemma below.
\begin{lemma}\label{stimaresto}
Let $N>2$ and $\theta=\theta(t)$ be a global classical solution to \eqref{main_eq}-\eqref{IC}. If
\begin{equation*}
\sigma_\tau(t):=\int_{t-\tau}^t \max_{1\leq k \leq N} \vert {\dot \omega}_k(s)\vert \, ds,
\end{equation*}
then the velocity diameter $D(\omega(\cdot))$ satisfies
\begin{equation*}
\md^+ D(\omega (t))\leq 2 \kappa \sigma_{\tau} (t)-\kappa\,\zeta_* D(\omega (t))\quad \mbox{for all} \quad t\ge t_*,
\end{equation*}
where $\zeta_*, t_*$ are the positive constants in \eqref{estimate_below}.
\end{lemma}
\begin{proof}
As in the proof of Lemma \ref{lem_ineq_d}, due to the continuity of the trajectories $\omega_i(t),$ $i=1, \dots, N,$  there is an at most countable system of open disjoint
intervals $\{\mathcal{I}_\sigma\}_{\sigma\in\N}$ such that
$$
   \bigcup_{\sigma\in\N} \overline{\mathcal{I}_\sigma} = [0,\infty),
$$
and  for each ${\sigma\in\N}$ there exist indices $i(\sigma)$, $j(\sigma)$
such that
\bq\label{dom}
   D(\omega(t)) = \omega_{i(\sigma)}(t) - \omega_{j(\sigma)}(t), \quad t\in \mathcal{I}_\sigma.
\eq
For simplicity of notation we can  put  $i:=i(\sigma)$, $j:=j(\sigma).$ Of course, we can assume $i\ne j.$ For $t\in \mathcal{I}_\sigma$, we have

\begin{align}\label{C2}
\begin{aligned}
&\frac{1}{2}\md^+ D(\omega(t))^2\cr
&\quad =\lt( \omega_i(t) -\omega_j(t)\rt)\Bigg( \frac \kappa N\sum_{k \ne i} (\cos (\theta_k(t-\tau)-\theta_i(t)) (\omega_k(t-\tau) -\omega_i(t))\\
& \hspace{4cm}- \frac \kappa N\sum_{k \ne j} (\cos (\theta_k(t-\tau)-\theta_j(t)) (\omega_k(t-\tau) -\omega_j(t))\Bigg)  \\
\hspace {0.5 cm}
&\quad = \frac \kappa N \sum_{k \ne i} \cos (\theta_k(t-\tau)-\theta_i(t)) \lt( \omega_i(t)-\omega_j(t)\rt)\lt( \omega_k(t-\tau) -\omega_i(t) \rt) \\
&\qquad  -\frac \kappa N\sum_{k \ne j} \cos (\theta_k(t-\tau)-\theta_j(t))\lt( \omega_i(t)-\omega_j(t)\rt)\lt( \omega_k(t-\tau) -\omega_j(t) \rt)  \\
&\quad = I_1+I_2.
\end{aligned}
\end{align}
Now, we can rewrite
\begin{align}\label{CI1}
\begin{aligned}
I_1 &= \frac \kappa N\sum_{k \ne i} \cos (\theta_k(t-\tau)-\theta_i(t)) \lt( \omega_i(t)-\omega_j(t)\rt)\lt( \omega_k(t-\tau) -\omega_k(t)\rt) \cr
&\quad + \frac  \kappa  N\sum_{k \ne i} \cos (\theta_k(t-\tau)-\theta_i(t)) \lt( \omega_i(t)-\omega_j(t)\rt) \lt(\omega_k(t)-\omega_i(t)\rt),\cr
I_2 &= -\frac \kappa N\sum_{k \ne j}  \cos (\theta_k(t-\tau)-\theta_j(t))\lt(\omega_i(t)-\omega_j(t)\rt)\lt(\omega_k(t-\tau) -\omega_k(t)\rt)\cr
&\quad -\frac \kappa N \sum_{k \ne j} \cos (\theta_k(t-\tau)-\theta_j(t)) \lt( \omega_i(t)-\omega_j(t)\rt)\lt( \omega_k(t)-\omega_j(t)\rt).
\end{aligned}
\end{align}
We observe that
\begin{equation*}
\omega_k(t) \leq \omega_i(t) \quad \mbox{and} \quad \omega_k(t) \geq \omega_j(t) \quad \mbox{for all} \quad k=1, \dots, N,
\end{equation*}
due to \eqref{dom}. Hence, using \eqref{estimate_below}  in \eqref{CI1}, we obtain
\begin{equation}\label{C10}
I_1 \leq \frac \kappa N  {D(\omega(t))}  \sum_{k=1}^N |\omega_k(t-\tau) -\omega_k(t)|+\frac \kappa N \zeta_* \sum_{k\ne i} \lt( \omega_i(t)-\omega_j(t)\rt)\lt(\omega_k(t)-\omega_i(t)\rt).
\end{equation}
Similarly, we also estimate
\begin{equation}\label{C11}
I_2 \leq \frac \kappa N D(\omega(t))\sum_{k=1}^N|\omega_k(t-\tau) -\omega_k(t)|-\frac \kappa N \zeta_*\sum_{k\ne j} \lt( \omega_i(t)-\omega_j(t)\rt)\lt( \omega_k(t)-\omega_j(t)\rt).
\end{equation}
Hence, using \eqref{C10} and \eqref{C11} in \eqref{C2}, we obtain
$$
\frac 1 2 \md^+ D(\omega(t))^2\leq 2 \frac{\kappa}N D(\omega(t)) \sum_{k=1}^N |\omega_k(t-\tau)-\omega_k(t)| -\kappa \zeta_* D(\omega(t))^2,
$$
and thus
\begin{equation}\label{C12}
\md^+ D(\omega(t))\leq 2 \frac \kappa  N \sum_{k=1}^N |\omega_k(t-\tau)-\omega_k(t)| -\kappa \zeta_* D(\omega(t)).
\end{equation}
One can estimate
$$\begin{aligned}
\sum_{k=1}^N |\omega_k(t-\tau) -\omega_k(t)|\leq \sum_{k=1}^N \int_{t-\tau }^t |\dot{\omega}_k(s)| \,ds \leq  N\sigma _{\tau} (t),
\end{aligned}$$
and so from \eqref{C12} we obtain
$$
\md^+ D(\omega(t)) \le 2 \kappa\sigma_{\tau} (t)-\kappa \zeta_* D(\omega(t)),
$$
which completes the proof.
\end{proof}
We then show that the time derivative of $\omega_i$ can be bounded from above by the sum of $\sigma_\tau$ and $D(\omega(t))$, which allows us to construct an appropriate Lyapunov functional for the complete frequency synchronization.
\begin{lemma}\label{altro}
Let $N>2$ and $\theta=\theta(t)$ be a global classical solution to \eqref{main_eq}-\eqref{IC}. Then we have
\begin{equation}\label{C15}
\max_{1 \leq i \leq N} |\dot{\omega}_i(t)| \le \kappa\sigma_\tau (t)+\kappa D(\omega(t)) \quad \mbox{for all} \quad t\ge t_*.
\end{equation}
\end{lemma}
\begin{proof}
It follows from \eqref{eq_omega} that
$$\begin{aligned}
|\dot{\omega}_i(t)| &= \Bigg| \frac \kappa N \sum_{k \ne i} \cos(\theta_k(t-\tau)-\theta_k(t)) (\omega_k(t-\tau) -\omega_k(t)) \cr
&\hspace{3cm}+\frac \kappa N \sum_{k\ne i} \cos(\theta_k(t-\tau)-\theta_k(t)) (\omega_k(t) -\omega_i(t))\Bigg|\cr
&\le\frac \kappa N\sum_{k=1}^N  |\omega_k(t-\tau) -\omega_k(t)|+\frac \kappa N\sum_{k=1}^N  |\omega_k(t)-\omega_i(t)| \cr
&\le \frac{\kappa}{N} \sum_{k=1}^N \int_{t-\tau }^t |\dot{\omega}_k(s)| \,ds +k D(\omega(t)) \cr
&\le  \kappa \sigma _{\tau} (t)+\kappa D(\omega(t)).
\end{aligned}$$
Taking the maximum for $i\in \{1, \dots, N\},$ we conclude \eqref{C15}.
\end{proof}
We now provide the details of our main result on the exponential complete frequency synchronization estimate.

\begin{proof}[Proof of Theorem \ref{thm_main}]
Let us define the Lyapunov functional
\[
\mathcal L(t)=D(\omega(t))+\eta \int_{t-\tau }^t e^{-(t-s)}{\int_s^t \max_{1 \leq j \leq N} |\dot{\omega}_j(\sigma )|\,d\sigma} \,ds.
\]
Then, from Lemma \ref{stimaresto} and Lemma \ref{altro} we have
$$\begin{aligned}
\md^+ \mathcal L(t) &=\md^+ D(\omega(t))-\eta e^{-\tau} \int_{t-\tau }^t \max_{1 \leq j \leq N} |\dot{\omega}_j(\sigma)|\,d\sigma \\
&\quad +\eta \int_{t-\tau}^t e^{-(t-s)} \max_{1 \leq j \leq N} |\dot{\omega}_j (t)| \,ds -\eta \int_{t-\tau }^t e^{-(t-s)}\int_s^t \max_{1 \leq j \leq N} |\dot{\omega}_j(\sigma )|\, d\sigma\, ds \\
&\le  \lt(-\kappa\zeta_* + \kappa\,\eta\,(1-e^{-\tau} )\rt) D(\omega(t))+ \lt(2\kappa -\eta e^{-\tau}+\kappa\, \eta\, (1-e^{-\tau})\rt) \sigma_{\tau} (t)\\
&\quad - \eta \int_{t-\tau}^t e^{-(t-s)} \int_s^t \max_{1 \leq j \leq N} |\dot{\omega}_j(\sigma)|\,d\sigma\, ds, \quad t\ge t_*.
\end{aligned}$$
Now, we want to show that if $\tau$ is sufficiently small one can choose the positive parameter $\eta$ in the definition of  the Lyapunov functional $\mathcal L(\cdot)$ such that previous estimate implies
\begin{equation}\label{C18}
\md^+\mathcal L(t)\le - \gamma \mathcal L(t) \quad \mbox{for all} \quad t\ge t_*,
\end{equation}
for a suitable positive constant $\gamma.$
In order to obtain \eqref{C18}, the parameters $\kappa, \eta$, and $\tau$ have to satisfy
\begin{equation}\label{eta1}
{2\kappa }-\eta\,e^{-\tau}+\kappa\,\eta\, (1-e^{-\tau})\le 0
\end{equation}
and
\begin{equation}\label{eta2}
-\kappa\,\zeta_* +\kappa\,\eta\, (1-e^{-\tau})<0.
\end{equation}
The inequality  \eqref{eta1} holds if and only if
\begin{equation}\label{C30}
\eta \ge \frac {2\kappa }{e^{-\tau }-\kappa (1-e^{-\tau })}.
\end{equation}
This implies a first restriction on the time delay size, namely
\begin{equation}\label{C31}
\tau <\ln \left ( 1+  \frac{1} \kappa
\right ).
\end{equation}
Condition \eqref{eta2} instead implies
\begin{equation}\label{C32}
\eta <\frac {\zeta_*}{1-e^{-\tau }}.
\end{equation}
Then, in order to choose a parameter  $\eta$ in the definition of $\mathcal L(\cdot)$ satisfying both \eqref{C30} and \eqref{C32}, we need
$$
\frac {2\kappa }{e^{-\tau }-\kappa (1-e^{-\tau })}< \frac {\zeta_*}{1-e^{-\tau }},
$$
and this gives a further condition on $\tau,$ namely
$$ \tau <\ln \left (
1+ \frac {\zeta_*} {\kappa (2+\zeta_*)}
\right ),
$$
which is stronger than \eqref{C31}. Thus, the theorem is proved.
\end{proof}

%%%%%%%%%%%%%%%%%%%%%%%%%%%%%%%%%%%%%%%%%%%%%%%%%%%%%%%%%%%%%%%%%%%%%%%%%%%%%%%%%
%
%
%
%
%
%%%%%%%%%%%%%%%%%%%%%%%%%%%%%%%%%%%%%%%%%%%%%%%%%%%%%%%%%%%%%%%%%%%%%%%%%%%%%%%%%

%%%%%%%%%%%%%%%%%%%%%%%%%%%%%%%%%%%%%%%%%%%%%%%%%%%%%%%%%%%%%%%%%%%%%%%%%%%%%%%%%
%
%
%                        Acknowledgments
%
%
%%%%%%%%%%%%%%%%%%%%%%%%%%%%%%%%%%%%%%%%%%%%%%%%%%%%%%%%%%%%%%%%%%%%%%%%%%%%%%%%%

\section*{Acknowledgments}
Y.-P. Choi was supported by POSCO Science Fellowship of POSCO TJ Park Foundation. C. Pignotti was supported by GNAMPA-INdAM and RIA-UNIVAQ.

%%%%%%%%%%%%%%%%%%%%%%%%%%%%%%%%%%%%%%%%%%%%%%%%%%%%%%%%%%%%%%%%%%%%%%%%%%%%%%%%%
%
%
%                         thebibliography
%
%
%%%%%%%%%%%%%%%%%%%%%%%%%%%%%%%%%%%%%%%%%%%%%%%%%%%%%%%%%%%%%%%%%%%%%%%%%%%%%%%%%


\begin{thebibliography}{10}
\bibitem{ABPRS} J. A. Acebr\'on, L. L. Bonilla, C. J. P. P\'erez Vicente, F. Ritort, and R. Spigler, The Kuramoto model: A simple paradigm for synchronization phenomena, Rev. Mod. Phys., 77, (2005), 137--185.
\bibitem{BCM15} D. Benedetto, E. Caglioti and U. Montemagno, On the complete phase synchronization of the Kuramoto model in the mean-field limit, Comm. Math. Sci., 13, (2015), 1775--1786.
\bibitem{BB} J. Buck, E. Buck, Biology of synchronous flashing of fireflies, Nature, 211, (1966), 562.
\bibitem{CCHKK14} J. A. Carrillo, Y.-P. Choi, S.-Y. Ha, M.-J. Kang, and Y. Kim, Contractivity of transport distances for the kinetic Kuramoto equation, J. Stat. Phys., 156, (2014), 395--415.
\bibitem{CHJK12} Y.-P. Choi, S.-Y. Ha, S. Jung, and Y. Kim, Asymptotic formation and orbital stability of phase-locked states for the Kuramoto model, Physica D, 241, (2012), 735--754.
\bibitem{CHKK13} Y.-P. Choi, S.-Y. Ha, M. Kang, and M. Kang, Exponential synchronization of finite-dimensional Kuramoto model at the critical coupling strength, Commun. Math. Sci., 11, (2013), 385--401.
\bibitem{CH17} Y.-P. Choi and J. Haskovec, Cucker-Smale model with normalized communication weights and time delay, Kinetic and Related Models, 10, (2017), 1011--1033.
\bibitem{CPP} Y.-P. Choi, A. Paolucci, and C. Pignotti, Consensus of the Hegselmann-Krause opinion formation model with time delay, preprint ArXiv:1909.02795.
\bibitem{CP19} Y.-P. Choi and C. Pignotti, Emergent behavior of Cucker-Smale model with normalized weights and distributed time delays, Networks and Heterogeneous Media, to appear.
\bibitem{CS09} N. Chopra and M. W. Spong, On exponential synchronization of Kuramoto oscillators, IEEE Trans. Autom. Control, 54, (2009), 353--357.
\bibitem{DB11} F. Dorfler, F. Bullo, On the critical coupling for Kuramoto oscillators, SIAM J. Appl. Dyn., 10, (2011) 1070--1099.
\bibitem{HKR16} S.-Y. Ha, H. K. Kim and S. W. Ryoo, Emergence of phase-locked states for the Kuramoto model in a large coupling regime, Comm. Math. Sci., 14, (2016), 1073--1091.
\bibitem{Ku75} Y. Kuramoto, International symposium on mathematical problems in mathematical physics, Lecture Notes Phys., 39, (1975), pp. 420.
\bibitem{N01} S.-I. Niculescu, Delay Effects on Stability: A Robust Control Approach. New York: Springer-Verlag, 2001, vol. 269, Lecture Notes in Control and Information Sciences.
\bibitem{PJM10} A. Papachristodoulou, A. Jadbabaie, and U. M\"unz, Effects of Delay in Multi-Agent Consensus and Oscillator Synchronization, IEEE Trans. Autom. Control, 55, (2010), 1471--1477.
\bibitem{PLR} A. Pluchino, V. Latora, and A. Rapisarda, Changing opinions in a changing world: a new
perspective in sociophysics Int. J. Mod. Phys. C, 16, (2005), 515--531.
\bibitem{SPMA12} G. S. Schmidt, A. Papachristodoulou, U. M\"unz, and F. Allg\"ower, Frequency synchronization and phase agreement in Kuramoto oscillator networks with delays, Automatica, 48, (2012), 3008--3017.
\bibitem{Ward} J. B. Ward, Equivalent circuits for power-flow studies, Trans. Am. Inst. Electr. Eng., 68, (2009), 373--382.
\bibitem{YS99} M. K. S. Yeung and S. H. Strogatz, Time delay in the Kuramoto model of coupled oscillators, Phys. Rev. Lett., 82, (1999), 648--651.
\end{thebibliography}
\end{document}